\documentclass{article}
\usepackage{graphicx} 
\usepackage{url}
\usepackage{amssymb}
\usepackage{hyperref}
\usepackage{url}
\usepackage{lineno}
\usepackage{amsmath,amsthm,amssymb}
\usepackage{thmtools}
\usepackage{physics}
\usepackage{mathtools}
\usepackage{mathrsfs}
\usepackage{indentfirst}
\usepackage{bbm}
\newcommand{\allones}{\mathbbm{1}}
\newcommand{\identity}{\mathbb{I}}

\newtheorem{theorem}{Theorem}[section]
\newtheorem{lem}[theorem]{Lemma}
\newtheorem{cor}[theorem]{Corollary}
\newtheorem{prop}[theorem]{Proposition}

\theoremstyle{definition}			                						
\newtheorem{definition}[theorem]{Definition}
\newtheorem{rem}[theorem]{Remark}

\usepackage{arxiv}

\usepackage[utf8]{inputenc} 
\usepackage[T1]{fontenc}    
\usepackage{hyperref}       
\usepackage{url}            
\usepackage{booktabs}       
\usepackage{amsfonts}       
\usepackage{nicefrac}       
\usepackage{microtype}      
\usepackage{lipsum}
\usepackage{graphicx}
\graphicspath{ {./images/} }

\title{Some Spectral Properties of the $B_{\alpha}$ Matrix of a Graph}

\author{
 Garrett Kepler \\
  Department of Mathematics \& Statistics \\
  Washington State University\\
  Pullman, WA, USA \\
  \texttt{garrett.kepler@wsu.edu} \\
}

\begin{document}
\maketitle
\begin{abstract}
Recently, Samanta \emph{et al.} introduced a unifying perspective to common graph-associated matrices through the matrix $B_{\alpha}(G)=\alpha A(G)+(1-\alpha)L(G)$ where $\alpha\in [0,1]$. We resolve an open problem regarding the positive semi-definiteness of $B_{\alpha}$ and similar matrices. We also explore eigenvector properties of $B_{\alpha}$ through edge and vertex removal. As a consequence, we provide a basis for results on the common graph-associated matrices.
\end{abstract}

\section{Introduction}
Let graph $G=(V(G),E(G))$ be a graph with vertex set $V(G)$ and edge set $E(G)$ such that $|V(G)|=n$. If an edge $\{u,v\}$ exists in $E(G)$ for vertices $u,v\in V(G)$, we say they are \emph{adjacent} (not adjacent) and denote this relationship $u\sim v$ ($u\nsim v$). Throughout, we will assume $G$ is undirected $(u\sim v \implies v\sim u)$ and simple $(u\nsim u)$. The \emph{degree} of a vertex $v$ in a graph $G$, $d_G(v)$, is the number of vertices adjacent to $v$. The minimum degree of a graph will be denoted $\delta(G)$. Let $M_n$ be the set of $n\times n$ matrices. For a matrix $M\in M_n$ to \emph{admit} an eigenvector $x\in \mathbb{R}^n$, we mean $Mx=\lambda x$ for some scalar $\lambda$. In the following sections, we will refer to several matrices defined below.
\begin{definition}[Adjacency Matrix]
    For a graph $G$, we define the adjacency matrix, $A=A(G)\in M_n$, such that
    \begin{align*}
        [A]_{ij}&=\begin{cases}
            1~~\text{if $i\sim j$}\\
            0~~\text{otherwise}
        \end{cases}
    \end{align*}
\end{definition}

\begin{definition}[Degree Matrix]
The degree matrix, $D=D(G)\in M_n$, is the diagonal matrix with diagonal entries corresponding to vertex degrees. That is,
    \begin{align*}
        [D]_{ij} &= \begin{cases}
            d_G(i)~~\text{if }i=j\\
            0~~\text{otherwise}
        \end{cases}
    \end{align*}
\end{definition}

\begin{definition}[Laplacian Matrix]
    The Laplacian matrix, $L=L(G)\in M_n$, is defined as $L=D-A$.
\end{definition}
\begin{definition}[Normalized Matrices]
    The normalized adjacency matrix, $\mathcal{A}=\mathcal{A}(G)\in M_n$, is defined as $\mathcal{A}=D^{-1/2}AD^{-1/2}$. Likewise, the normalized Laplacian matrix is defined as $\mathcal{L}=\identity-\mathcal{A}$ where $\identity$ denotes the $n\times n$ identity matrix.  
\end{definition}

Lastly, we will prove results about the following matrices and use them to prove results about the Laplacian and adjacency matrix. 
\begin{definition}[$\alpha$-Matrices]
    We define $A_{\alpha}(G)= \alpha D(G)+(1-\alpha)A(G)$ and $B_{\alpha}(G)= \alpha A(G)+(1-\alpha)L(G)$ where $\alpha\in [0,1]$.
\end{definition}
The matrices above are indeed special cases of the $\alpha$-matrices. That is, the adjacency, Laplacian, and degree matrix have the property that $L(G)=B_0(G)=A_1(G)-A_0(G),~A(G) = B_1(G)=A_0(G),~$ and $D(G)=2B_{\frac{1}{2}}(G)=A_1(G)$ respectively. Introduced in \cite{sam24}, $B_{\alpha}$ has many useful applications in general theories of graph-associated matrices. Here, we will use $B_{\alpha}$ to obtain general results about eigenvectors of graphs. Moreover, will prove conditions on $\alpha$ that guarantee $A_{\alpha}$ and $B_{\alpha}$ are  positive semi-definite (PSD). For conciseness, we will denote $B_{\alpha}$ as $B$ and use subscripts when $\alpha$ is to take on a specific value.

As noted in the seminal paper \cite{sam24}, we can reformulate $B$ in several equivalent ways. For example:
\begin{align*}
    B&=\alpha A+(1-\alpha)L\\
    &=(2\alpha - 1)A+(1-\alpha)D\\
    &= (1-2\alpha)L+\alpha D
\end{align*}
Clearly, for $\alpha=0$, $B$ inherits positive semi-definiteness from the Laplacian. Contrarily, for $\alpha=1$, $B$ is not-PSD for a graph with at least one edge. In fact, Samanta \emph{et al.} were able to achieves bounds on when $B$ is positive definite: 
\begin{cor}[Corollary 3.2 of \cite{sam24}]\label{posidef}
    If $G$ has no isolated vertices, $B_\alpha$ is PSD for all $\alpha\in (0,\frac{2}{3})$.
\end{cor}
Likewise, Nikiforov was able to prove bounds for $A_{\alpha}$ in \cite{niki17}:
\begin{prop}[Proposition 6 of \cite{niki17}]
    $A_{\alpha}$ is PSD for all $\alpha\in (\frac{1}{2},1]$ and positive definite for all $\alpha\in (\frac{1}{2},1]$ when the graph has no isolated vertices.
\end{prop}
In both works, the authors proposed the problem of finding an $\alpha\in(0,1)$ where a transition between PSD and not-PSD occurs for these matrices. We resolve these in the following section.
\section{Positive Semi-Definiteness of $\alpha$-Matrices}

To obtain our result, we utilize the smallest eigenvalue of the normalized adjacency matrix $\mathcal{A}$. Coming from Equation 1.6 and Lemma 1.7 of \cite{chung97}, the following statement will be foundational.
\begin{lem}[\cite{chung97}]\label{chung}
For the normalized adjacency matrix $\mathcal{A}$ and normalized Laplacian $\mathcal{L}$ of $G$, we have
    \begin{align*}
    \lambda_n(\mathcal{A})&=\min_{\substack{x\neq 0\\
        ||x||=1}}x^T\mathcal{A}x=\min_{y\neq 0}\frac{y^TAy}{y^TDy}=1-\mu_1(\mathcal{L})=1-\max_{\substack{x\neq 0\\
        ||x||=1}}x^T\mathcal{L}x
    \end{align*}
Moreover, for a graph $G$ with no isolated vertices, we have $\frac{n}{n-1}\leq \mu_1(\mathcal{L})\leq 2$ and $-\frac{1}{n-1}\geq \lambda_n(\mathcal{A})\geq -1$
\end{lem}

\begin{theorem}\label{general}
    Let $a,b\neq 0$ and $G$ be an undirected graph such that $\delta(G)\geq 1$. Then, $\lambda_n(aD+bA)=0$ if and only if $\lambda_n(\mathcal{A})=\frac{-a}{b}$
\end{theorem}
\begin{proof}
    Since $G$ is undirected, $aD+bA$ is symmetric such that Rayleigh-Ritz guarantees the following:
    \begin{align*}
        \lambda_n(aD+bA)&=\min_{\substack{x\neq 0\\
        ||x||=1}} (ax^TDx+bx^TAx)\\
    \end{align*}
    Since $\delta(G)\geq 1$, $x^TDx> 0$ when $x\neq 0$. So,
    \begin{align*}
       \min_{\substack{x\neq 0\\
        ||x||=1}} (ax^TDx+bx^TAx)&=\min_{\substack{x\neq 0\\
        ||x||=1}}\left(1+\frac{bx^TAx}{ax^TDx}\right)ax^TDx\\
    \end{align*}
    Since, $ax^TDx\neq 0$, the minimum of the product is zero if and only if the minimum of $1+\frac{bx^TAx}{ax^TDx}$ is zero. Then, using Lemma \ref{chung}, we have our result:
    \begin{align*}
        \min_{\substack{x\neq 0\\
        ||x||=1}}\left(1+\frac{bx^TAx}{ax^TDx}\right)=0
        &\iff \min_{\substack{x\neq 0\\
        ||x||=1}}\frac{bx^TAx}{ax^TDx}=-1\\
        &\iff \min_{\substack{x\neq 0\\
        ||x||=1}}\frac{x^TAx}{x^TDx}=-\frac{a}{b}\\
        &\iff \lambda_n(\mathcal{A})=-\frac{a}{b}
    \end{align*}
\end{proof}
 Equipped with the above theorem, we can find where the transition from PSD to not-PSD will occur for certain choices of $a$ and $b$. For example, resolving Problem 8 of \cite{niki17}:
\begin{theorem}
    For an undirected graph $G$ with minimum degree 1, the minimum $\alpha\in (0,1)$ such that $A_{\alpha}$ is PSD is $\alpha_0(G)=\frac{\lambda_n(\mathcal{A})}{\lambda_n(\mathcal{A})-1}$
\end{theorem}
\begin{proof}
    First, since $\alpha>0$ and $1-\alpha>0$ for $\alpha\in (0,1)$, Theorem \ref{general} guarantees $\lambda_n(\alpha D+(1-\alpha)A)=0$ only when $\lambda_n(\mathcal{A})=\frac{\alpha}{\alpha-1}$. Second, as a consequence of Propositions 3 and 4 of \cite{niki17}, $\lambda_{n}(A_{\alpha})$ is non-decreasing and continuous in $\alpha$. Therefore, the minimum $\alpha$ such that $A_{\alpha}$ is PSD is $\alpha_0(G)=\frac{\lambda_n(\mathcal{A})}{\lambda_n(\mathcal{A})-1}$.
\end{proof}

Resolving Problem 2 of \cite{past25}:
\begin{theorem}
For an undirected graph $G$ with minimum degree 1, the maximum $\alpha \in (0,1)$ such that $B_{\alpha}$ is PSD is $\beta_0(G)=\frac{\lambda_n(\mathcal{A})-1}{2\lambda_n(\mathcal{A})-1}$
\end{theorem}
\begin{proof}
    By Corollary \ref{posidef}, $B_{\alpha}$ is positive definite for all $\alpha\in (0,\frac{2}{3})$. That is, $\lambda_n((1-\alpha)D+(2\alpha-1)A)>0$ for $\alpha\in (0,\frac{2}{3})$. Moreover, at $\alpha = 1$, since the minimum degree of $G$ is 1, there is at least one edge which implies $\lambda_n((1-1)D+(2\cdot 1-1)A)=\lambda_n(A)<0$. So, as a consequence of Theorem 3.1 of \cite{sam24}, the continuity of $\lambda_n(B_{\alpha})$ implies $\lambda_n(B_{\alpha})=0$ for some $\alpha \in [\frac{2}{3},1)$.
    
    Since $1-\alpha >0$ and $2\alpha -1>0$ for all $\alpha\in [\frac{2}{3},1)$, Theorem \ref{general} guarantees $\lambda_n((1-\alpha) D+(2\alpha-1)A)=0$ only when $\lambda_n(\mathcal{A})=\frac{\alpha-1}{2\alpha-1}$. Therefore, the maximum $\alpha$ such that $B_{\alpha}$ is PSD is $\beta_0(G)=\frac{\lambda_n(\mathcal{A})}{\lambda_n(\mathcal{A})-1}$.
\end{proof}

The two above theorems lead to the following solution to Problem 3 of \cite{past25}:
\begin{cor}
    For an undirected graph $G$ with minimum degree 1, $\epsilon(G)=\beta_0(G)-\alpha_0(G)=\frac{\lambda_n(\mathcal{A})-1}{2\lambda_n(\mathcal{A})-1}-\frac{\lambda_n(\mathcal{A})}{\lambda_n(\mathcal{A})-1}$
\end{cor}
Figure \ref{fig:onlyfigure} demonstrates these transitions.
\begin{figure}[h]
    \centering
        \includegraphics[width=.8\linewidth]{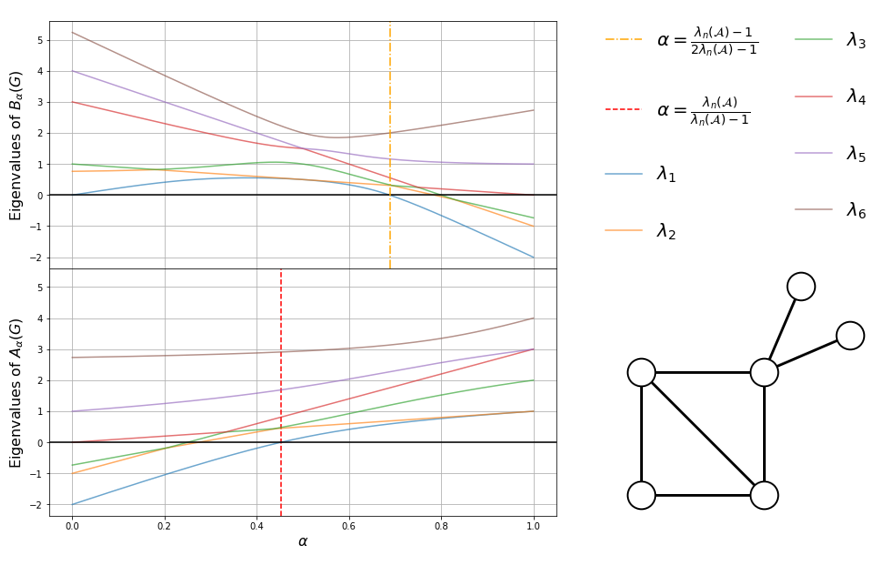}
    \caption{The transition from PSD to not-PSD is illustrated using the example graph from \cite{past25}. Top plot:  $B_{\alpha}$ eigenvalues as $\alpha$ varies. Bottom plot: $A_{\alpha}$ eigenvalues as $\alpha$ varies. Bottom right: the example graph.}
    \label{fig:onlyfigure}
\end{figure}

\section{Eigenvectors \& Graph Perturbation}

In this section, we address restrictions on eigenvectors of $B_{\alpha}$ that guarantee they are also eigenvectors after structural changes to graphs. First, we will cover an entry-wise view on the eigenequation:

\begin{prop}\label{entrywiseb}
    $[Bx]_i=\alpha \sum\limits_{j\sim i} x_j+(1-\alpha)\sum\limits_{j\sim i}x_i-x_j$ for all $x\in \mathbb{R}^n$
\end{prop}
\begin{proof}
    We know $[Bx]_i=\alpha[Ax]_i+(1-\alpha)[Lx]_i$. Using the entry-wise matrix multiplication formula, it follows from the definitions of $A$ and $L$ that
    \begin{align*}
        [Ax]_i = \sum_{j=1}^nA_{ij}x_j = \sum_{j\nsim i}0\cdot x_j+\sum_{j\sim i}1\cdot x_j=\sum_{j\sim i}x_j\\
        [Lx]_i=[Dx]_i-[Ax]_i=d_ix_i-\sum_{j\sim i}x_j=\sum_{j\sim i}x_i-x_j
    \end{align*}
    Substituting in to $B=\alpha A+(1-\alpha) L$, the result follows.
\end{proof}
\begin{lem}\label{lemmaineedrn}
    For $x\in \mathbb{R}^n$, $Bx=\lambda x$ if and only if $[Bx]_i=\lambda x_i$ for all $i$
\end{lem}
Thus, throughout this section, we will utilize the following corollary.
\begin{cor}\label{corineedrn}
    For $x\in \mathbb{R}^n$, $Bx=\lambda x$ if and only if $\alpha \sum\limits_{j\sim i} x_j+(1-\alpha)\sum\limits_{j\sim i}x_i-x_j=\lambda x_i$ for all $i$.
\end{cor}

\subsection{Perturbation of Edges}
Let $H$ be a subgraph of $G$ such that $E(H)\subseteq E(G)$ and $V(H)\subseteq V(G)$. Likewise, let $G-H$ be the subgraph of $G$ obtained by removing the edges $E(H)$ from $E(G)$. Note that $B(G)=B(H)+B(G- H)$.
\begin{theorem}\label{edgeremoval}
    Let $B(G)x=\lambda x$ and $\lambda'$ be some scalar. Then, the following are equivalent:
    \begin{enumerate}
        \item $B(H)x=\lambda'x$
        \item $B(G- H)x=(\lambda-\lambda')x$
        \item $\alpha \sum\limits_{\substack{j\sim i\\
        (i,j)\in E(H)}}x_j+(1-\alpha)\sum\limits_{\substack{j\sim i\\
        (i,j)\in E(H)}}x_i-x_j=\lambda' x_i$ for all $i$
        \item $\alpha \sum\limits_{\substack{j\sim i\\
        (i,j)\in E(G)\setminus E(H)}}x_j+(1-\alpha)\sum\limits_{\substack{j\sim i\\
        (i,j)\in E(G)\setminus E(H)}}x_i-x_j=(\lambda-\lambda')x_i$ for all $i$
    \end{enumerate}
\end{theorem}
\begin{proof}
    We know $B(G)=B(H)+B(G- H)$ implying $B(G)x=B(H)x+B(G- H)x$, or equivalently, $B(G)x-B(H)x=B(G- H)x$. If $B(H)x=\lambda' x$, then $B(G)x-B(H)x=\lambda x-\lambda' x=(\lambda-\lambda')x= B(G- H)$. Likewise, if $B(G- H)=(\lambda-\lambda')x=B(G)x-B(H)x$, then $\lambda x-B(H)x=0$. Thus, 1 is equivalent to 2. Considering Proposition \ref{entrywiseb}, we have the result.
\end{proof}

\begin{cor}\label{preserving}
    Let $B(G)x=\lambda x$. Then, $B(H)x=\lambda x$ if and only if $B(G-H)x=0$.
\end{cor}

In the case where $H$ is a single edge, we obtain further insight to the relationship between graph structure and eigenvector entry:

\begin{theorem}\label{singleedge}
    Let $B(G)x=\lambda x$ where $G$ is a non-empty, undirected graph $G$. Then, for any edge $\{u,v\}\in E(G)$,
    \begin{enumerate}
        \item $B(G-\{u,v\})x=\lambda x$ if and only if $(1-\alpha)(x_u-x_v)=-\alpha x_v=\alpha x_u$.
        \item $B(G-\{u,v\})x=\lambda'x$ where $\lambda'\neq \lambda$ if and only if $x_i=0$ for all $i\neq u,v$ and $x_u=-x_v\neq 0$. Moreover, the following must hold: 
        \begin{enumerate}
            \item $u$ and $v$ must have the same degree
            \item $\lambda'=(1-\alpha)(d_G(u)-1)=(1-\alpha)(d_G(v)-1)$
            \item $\lambda=(1-\alpha)(d_G(u)-1)+2-3\alpha=(1-\alpha)(d_G(v)-1)+2-3\alpha$
        \end{enumerate} 
    \end{enumerate}
\end{theorem}
\begin{proof}
    From Theorem \ref{edgeremoval} above, we have $B(G-\{u,v\})x=\lambda' x$ if and only if $B(\{u,v\})x=(\lambda-\lambda')x$. Using Proposition \ref{entrywiseb}, this is equivalent to $[B(\{u,v\})x]_i=(\lambda-\lambda')x_i$ for all $i$. 
    
    \begin{enumerate}
        \item Consider the case where $\lambda'=\lambda$. 
        
        By Corollary \ref{preserving}, $B(G-\{u,v\})x=\lambda x$ if and only if $B(\{u,v\})x=0$. For $i\neq u$ or $v$, $i$ is an isolated vertex in $\{u,v\}$ such that $[B(\{u,v\})x]_i=0$ holds trivially. So, we need only consider $u$ and $v$. For $i=u$ or $v$, $[B(\{u,v\})]_i=0$ is equivalent to to $\alpha x_v+(1-\alpha)(x_u-x_v)=0=\alpha x_u+(1-\alpha)(x_v-x_u)$ by Theorem \ref{edgeremoval}. Thus, the equivalence in 1 holds.
        
        \item Consider the case where $\lambda'\neq \lambda$. 
        
        For all $i\neq u$ or $v$, $[B(G-\{u,v\})]_i=[B(G)]_i$ since their adjacency is not affected. As such, since $[B(G)x]_i=\lambda x_i$, $[B(G-\{u,v\})]_i=\lambda'x_i$ for $i\neq u$ or $v$ if and only if $x_i=0$.
        
        Note, since $x_i=0$ for all other vertices, either $x_u\neq 0$ or $x_v\neq 0$ for $x$ to be an eigenvector. For $u$ and $v$, Theorem \ref{edgeremoval} implies $[B(\{u,v\})x]_i=(\lambda-\lambda')x_i$ if and only if 
        \begin{align*}
            (\lambda-\lambda')x_u&=\alpha x_v+(1-\alpha)(x_u-x_v)\\
            (\lambda-\lambda')x_v&=\alpha x_u+(1-\alpha)(x_v-x_u)
        \end{align*}
        which is equivalent to $(\lambda-\lambda'-\alpha)x_u=-(\lambda-\lambda'-\alpha)x_v$. So, $[B(\{u,v\})x]_i=(\lambda-\lambda')x_i$ for $i=u$ or $v$ if and only if $x_u=-x_v\neq 0$. Thus, the equivalence in 2 holds.\\

        \begin{enumerate}
            \item    Consider $[B(G-\{u,v\})]_u$. Since $x_j=0$ for every $j\sim u$ in $G-\{u,v\}$, we have
            \begin{align*}
            &[B(G-\{u,v\})]_u = \alpha\sum_{\substack{j\sim u\\ \text{in } (G-\{u,v\})}}x_j+(1-\alpha)\sum_{\substack{j\sim u\\ \text{in } (G-\{u,v\})}}x_u-x_j\\
            \implies& \lambda' x_u= \alpha\sum_{\substack{j\sim u\\ \text{in } (G-\{u,v\})}}0+(1-\alpha)\sum_{\substack{j\sim u\\ \text{in } (G-\{u,v\})}}x_u-0\\
            \implies& \lambda' x_u=(1-\alpha)(d_G(u)-1)x_u\\
        \end{align*}
        A similar calculation reveals $[B(G-\{u,v\})x]_v=(1-\alpha)(d_G(v)-1)x_v$. So, if $B(G-\{u,v\})x=\lambda'x$ and $x_u=-x_v\neq 0$, then $(1-\alpha)(d_G(u)-1)=(1-\alpha)(d_G(v)-1)$ implying $d_G(u)=d_G(v)$.
        \item Implicitly in proof above, $\lambda'=(1-\alpha)(d_G(u)-1)=(1-\alpha)(d_G(v)-1)$ is also shown.
        \item Note that $$[B(G)x]_u=[B(G-\{u,v\}]_u+\alpha x_v+(1-\alpha)(x_u-x_v)=\lambda'x_u-\alpha x_u+2(1-\alpha)x_u$$ such that $[B(G)x]_u=\lambda x_u=(\lambda'+2-3\alpha)x_u$. A similar argument shows $[B(G)x]_v=\lambda x_v=(\lambda'+2-3\alpha)x_v$. Note that $x_i=0$ for all $i\neq u,v$ implying $[B(G)x]_i=(\lambda'+2-3\alpha)x_i$ trivially holds. Thus, since $[B(G)x]_i=\lambda x_i$ for all $i$, we have $\lambda=\lambda'+2-3\alpha$ finishing the proof.
        \end{enumerate} 
    \end{enumerate}
\end{proof}
We can also consider the removal of all edges incident to a node. 
\begin{theorem}\label{fulledges}
    Let $N_v$ be the set of edges in a graph that contain a vertex $v$. Consider a vector $x\in \mathbb{R}^n$ such that $x^T\allones =0$ and $x_v=0$ for some $v\in V(G)$. If $v$ is adjacent to every vertex $u\in V(G)$ such that $x_u\neq 0$, then $B(G)x=\lambda x$ if and only if $B(G-N_v)x=(\lambda-1+\alpha)x$.
\end{theorem}
\begin{proof}
    Note that $v$ is adjacent to no vertex in the graph $G-N_v$. So, $[B(G-N_v)x]_v=0$. Consider $[B(G)x]_v$. 
    \begin{align*}
        [B(G)x]_v&=\alpha \sum_{j\sim v}x_j+(1-\alpha)\sum_{j\sim v}x_v-x_j\\
        &=\alpha \sum_{j=1}^nx_j+(1-\alpha)\sum_{j=1}^nx_j\\
        &=\alpha x^T\allones +(1-\alpha)x^T\allones\\
        &= 0 = [B(G-N_v)x]_v
    \end{align*}
    So, $[B(G)x]_v=\lambda x_v$ if and only if $[B(G-N_v)x]_v=(\lambda-1+\alpha)x_v$  trivially.
    
    Consider $[B(G-N_v)x]_i$ for $i\sim v$. 
    \begin{align*}
        [B(G-N_v)x]_i&=[B(G)x]_i-[B(N_v)x]_i\\
        &=[B(G)x]_i-\alpha x_v-(1-\alpha)(x_i-x_v)\\
        &=[B(G)x]_i-(1-\alpha)x_i
    \end{align*}
    As such, for all $i\sim v$, $[B(G)x]_i=\lambda x_i$ if and only if $[B(G-N_v)x]_i=(\lambda-1+\alpha)x_i$. Lastly, if $i\nsim v$, then $x_i=0$ by assumption such that Theorem $\ref{singleedge}$ guarantees we can add or remove the edge between $i$ and $v$ with no consequence on the eigenvalue or eigenvector. So, for all $i\in V(G)$, $[B(G)x]_i=\lambda x_i$ if and only if $[B(G-N_v)x]_i=(\lambda-1+\alpha)x_i$. The result follows by Lemma \ref{lemmaineedrn}.
\end{proof}

\subsection{Perturbation of Vertices}
\begin{theorem}\label{disconnectednode}
    Let $B(G)x=\lambda x$. Consider adding a vertex $v$ to $G$ with no additional edges to obtain $G\cup \{u\}$. Then, $x$ can be extended to an eigenvector $y=[x,0]^T$ of $B(G\cup \{u\})$ corresponding to $\lambda$.
\end{theorem}
\begin{proof}
    For $i\neq u$, no adjacency has changed from $G$ to $G\cup \{u\}$ such that $\lambda x_i=[B(G)x]_i=[B(G\cup \{u\})y]_i=\lambda x=y_i$. Since $y_v=0$ and $v$ is adjacent to no other vertices, $[B(G\cup \{u\})y]_v=\lambda y_v$ trivially. So, $[B(G\cup\{u\})y]_i$ for all $i\in V(G\cup \{u\})$ and by Lemma \ref{lemmaineedrn}, the result follows.
\end{proof}
\begin{theorem}\label{zeronode}
    Let $Bx=\lambda x$ such that $x^T\allones = 0$. Consider adding a vertex $u$ with edges between $u$ and every vertex in $v\in V(G)$ such that $x_v\neq 0$. If $G'$ is this new graph, then $x$ can be extended to an eigenvector $y = [x,0]^T$ of $B(G')$ corresponding to $\lambda -1+\alpha$.
\end{theorem}

\begin{proof}
    By Theorem \ref{disconnectednode}, we add a new vertex $u$ with no edges to obtain graph $G\cup \{u\}$ and we will have $B(G\cup \{u\})y=\lambda y$. By Theorem \ref{fulledges}, we can add an edge between $u$ and every node $v\in V(G)$ such that $x_v=0$ to obtain $G'$ and we will have $B(G')y=(\lambda-1+\alpha)y$.
\end{proof}

\subsection{Application to Laplacian \& Adjacency Matrices}
Again, note that the statements that apply to $B_{\alpha}$ hold for any matrix that $B_{\alpha}$ generalizes as well. To demonstrate the utility in this, we provide the following subsection.  

Firstly, we set $\alpha=0$ in Theorem \ref{singleedge} to obtain the following corollary:

\begin{cor}\label{lapsingle}
    Let $G$ be a non-empty, undirected graph such that $L(G)x=\lambda x$. Then, for any $\{u,v\}\in E(G)$,
    \begin{enumerate}
         \item $L(G-\{u,v\})x=\lambda x$ if and only if $x_u=x_v$.
        \item $L(G-\{u,v\})x=\lambda'x$ where $\lambda'\neq \lambda $ if and only if $x_i=0$ for all $i\neq u,v$ and $x_u=-x_v\neq 0$. Moreover, $u$ and $v$ have the same degree such that $\lambda=d_G(u)=d_G(v)$ and $\lambda'=d_G(u)-2=d_G(v)-2$.
    \end{enumerate}
\end{cor}
Likewise with $\alpha=1$ in Theorem \ref{singleedge}, we have the following:
\begin{theorem}\label{adjsingle}
Let $G$ be a non-empty, undirected graph where $A(G)x=\lambda x$. 

\begin{enumerate}
    \item For any $\{u,v\}\in E(G)$,
    \begin{enumerate}
        \item $A(G-\{u,v\})x=\lambda x$ if and only if $x_u=x_v=0$
        \item $A(G-\{u,v\})x=\lambda'x$ where $\lambda'\neq \lambda $ if and only if $x_i=0$ for all $i\neq u,v$ and $x_u=-x_v\neq 0$. Moreover, $\lambda'=0$ and $\lambda = -1$.
    \end{enumerate}
    \item Let $\lambda = \lambda_{max}(A(G))$ and $\{u,v\}\in E(G)$. Then, the following hold:
    \begin{enumerate}
        \item If either $x_u\neq 0$ or $x_v\neq 0$, then $x$ cannot be an eigenvector of $A(G-\{u,v\})$.
        \item If $G$ is connected, then $x$ cannot be an eigenvector of $A(G-\{u,v\})$.
    \end{enumerate}
\end{enumerate}
\end{theorem}

\begin{proof}
\begin{enumerate}
    \item Simply Theorem \ref{singleedge} with $\alpha=1$.
    \item The Perron-Frobenius theorem for non-negative matrices states $x$ must non-negative in general and positive if the matrix is also irreducible [Theorem 8.4.4 of \cite{john85}]. By 1, to preserve $x$ we need either $x_u=x_v=0$ or $x_u=-x_v$. As such, the case where $x_u=-x_v$ is impossible. For (a), the case where $x_u=x_v=0$ is the only way the eigenvector is still an eigenvector of $A(G-\{u,v\})$. For (b), if $G$ is connected, the matrix $A(G)$ is also irreducible which forces $x$ to be positive making $x_u=x_v=0$ impossible as well.
\end{enumerate}
    
\end{proof}
In Figure \ref{fig:demonstration1}, we demonstrate the above theorems on the example graph from \cite{past25}. 
\begin{figure}[h]
    \centering
    \includegraphics[width=0.28\textwidth]{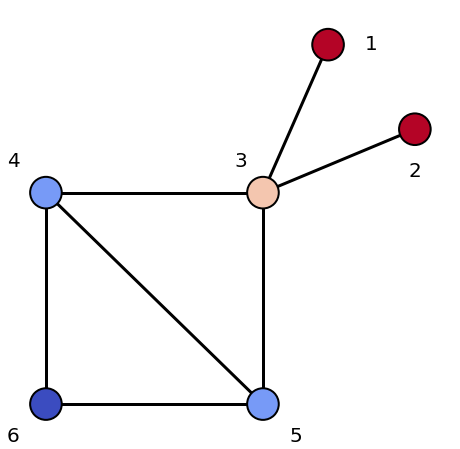}~~~~~~~~
    \includegraphics[width=.34\textwidth]{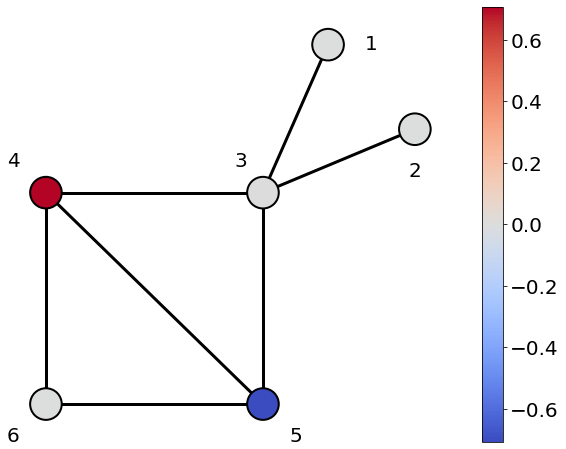}
    \caption{An example of \ref{lapsingle} and \ref{adjsingle} in action. Left: eigenvector $x\approx [ 0.51,0.12,0.51,-0.32,-0.32,-0.51]^T$ corresponding to $\lambda(L(G))\approx0.76$. Right: eigenvector $x\approx [0, 0, 0, 0.7,-0.7, 0]^T$ corresponding to eigenvalue $\lambda(A(G))= -1$. On left, adding edge $\{1,2\}$ or removing edge $\{4,5\}$ will not affect eigenvalue or eigenvector. On right, adding any edges among zero-assigned vertices will not affect eigenvalue or eigenvector while removing $\{4,5\}$ will only change eigenvalue by 1.}
    \label{fig:demonstration1}
\end{figure}

\section{Eigenvectors \& Graph Structure}
Theorem \ref{singleedge} gives us insight into the eigenvectors of structurally similar graphs. But, there are strict conditions that must be met for eigenvectors to remain eigenvectors following perturbation. This raises a couple questions. Mainly, what graphs admit this form of eigenvector? Secondly, how many graphs are there? We address these in part.
\begin{theorem}
        The graphs admitting the eigenvectors $x\in \mathbb{R}^n$ with all zeroes except for two distinct entries $x_u=-x_v\neq 0$ are formed by adding vertices and edges to a $K_2$ such that each other vertex is either adjacent to both vertices or neither vertices in the $K_2$.
\end{theorem}
\begin{proof}
    Without loss of generality, we can prove the result for $x_u=1=-x_v$.
    
    $\rightarrow$ Let $G$ be a graph admitting such an eigenvector $x$ such that $x_u=-x_v\neq 0$ are the only nonzero entries. By Theorem \ref{singleedge}, we can remove any edges between vertices assigned zero in $x$ such that the only edges left are incident with $u$ or $v$ such that $\lambda$ and $x$ are still an eigenvalue and eigenvector. That is, if we remove all edges between vertices assigned zero in $x$, $B(G')x=\lambda x$ where $G'$ is the new graph. Consider $[B(G')]_i$ for $i\neq u$ or $v$.
    \begin{align*}
        [B(G')x]_i&=\alpha \sum_{j=u,v}x_j+(1-\alpha)\sum_{j=u,v}-x_j\\
        &=(2\alpha-1)\sum_{j=u,v}x_j
    \end{align*}
    Since $[B(G')x]_i=\lambda x_i$ for all $i$, for $i\neq u$ or $v$
    \begin{align*}
       (2\alpha-1)\sum_{j=u,v}x_j&=\lambda x_i=\lambda 0=0
    \end{align*}
        This implies the sum is either empty such that $i\nsim u$ and $i\nsim v$ or non-empty such that $x_u=1$ appears as well as $x_v=-1$. As such, Theorems \ref{zeronode} and \ref{disconnectednode} guarantee we can remove each such vertex. The resulting graph is a $K_2$.\\

    $\leftarrow$ Note that $y\in \mathbb{R}^2$ with one entry $y_u=1$ and another $y_v=-1$ is an eigenvector of $B(\{u,v\})$ corresponding to $2-3\alpha$:
    \begin{align*}
        [B(\{u,v\})y]_u&= \alpha y_v+(1-\alpha)(y_u-y_v)=\alpha+2(1-\alpha)=2-3\alpha
    \end{align*}
    Likewise, $[B(\{u,v\})y]_v=-(2-3\alpha)y_v$. By Theorem \ref{zeronode}, we can add $n-2$ vertices adjacent to $u$ and $v$ extending $y$ with $n-2$ zeroes to an eigenvector $x = [y, 0,\dots, 0]^T$ of the new graph. By Theorem \ref{fulledges}, if we remove edges between a new vertex and $u$ or $v$, we must remove both edges. Likewise, for each of the new vertices, they are all assigned zero in $x$ such that \ref{singleedge} guarantees we can add any edges between these zero-assigned vertices and $x$ will be an eigenvector upon edge addition. So, $x$ will be an eigenvector of any graph obtainable by adding vertices and edges to $K_2$ such that each is adjacent to both or neither of the nodes in the $K_2$.
\end{proof}    
\begin{theorem}\label{proportion}
    The proportion of graphs admitting eigenvectors $x\in \mathbb{R}^n$ with all zeroes except for two distinct entries $x_u=-x_v\neq 0$ approaches 1 as $n\to \infty$.
\end{theorem}
\begin{proof}
    Without considering isomorphism, there are $2^{\binom{n}{2}}$ total undirected graphs on $n$ vertices. Using the theorem above, we will count how many of these graphs admit eigenvectors $x$ of the form described. 

    There are $\binom{n}{2}$ ways to select which vertices $u$ and $v$ will be assigned nonzero entries. There are 2 ways to select whether these two vertices are adjacent to each other or not. There are $2^{\binom{n-2}{2}}$ possible graphs on the vertices that are not $u$ and $v$. Each vertex that is not $u$ and $v$ must be either adjacent to both $u$ and $v$ or neither. So, there are 2 possible options for each of the $n-2$ other vertices giving $2^{n-2}$ ways to form the adjacency between the $n-2$ other vertices and the vertices $u$ and $v$. So, without considering isomorphism, in total we have $\binom{n}{2}\cdot 2\cdot 2^{\binom{n-2}{2}}2^{n-2}=2^{\log_2\binom{n}{2}+\binom{n-2}{2}+n-1}$ graphs that have eigenvectors of the form described above. So, it suffices to show that
    \begin{align*}
        \lim_{n\to \infty}\frac{1}{\binom{n}{2}}\left(\log_2\binom{n}{2}+\binom{n-2}{2}+n-1\right)=1
    \end{align*}
    Consider the limits of each summand individually. Firstly,
    \begin{align*}
        \lim_{n\to \infty}\frac{n-1}{\binom{n}{2}}=\lim_{n\to \infty}\frac{(n-1)2!(n-2)!}{n!}=\lim_{n\to\infty}\frac{2}{n}=0\\
    \end{align*}
    Secondly,
    \begin{align*}
        \lim_{n\to \infty}\frac{\binom{n-2}{2}}{\binom{n}{2}}=\lim_{n\to\infty}\frac{(n-2)!}{(n-4)!2!}\frac{(n-2)!2!}{n!}=\lim_{n\to\infty}\frac{(n-2)(n-3)}{n(n-1)}=1
    \end{align*}
    Lastly, consider $\frac{\log_2(x)}{x}=\frac{\ln(x)}{\ln(2)x}$ as $x\to \infty$. By L'Hopital's Rule,
    \begin{align*}
        \lim_{x\to\infty}\frac{\ln(x)}{x\ln(2)}=\lim_{x\to\infty}\frac{1/x}{\ln(2)} = 0
    \end{align*}
    Since this limit approaches 0, the subsequence produced as $n\to\infty$ also converges to 0:
    \begin{align*}
        \lim_{n\to\infty} \frac{\log_2\binom{n}{2}}{\binom{n}{2}}=0
    \end{align*}
    Therefore, since each each of these limits exists,
    \begin{align*}
        \lim_{n\to \infty}\frac{1}{\binom{n}{2}}\left(\log_2\binom{n}{2}+\binom{n-2}{2}+n-1\right)&=\lim_{n\to \infty}\frac{\log_2\binom{n}{2}}{\binom{n}{2}}+\lim_{n\to\infty}\frac{\binom{n-2}{2}}{\binom{n}{2}}+\lim_{n\to\infty}\frac{n-1}{\binom{n}{2}}\\
        &= 0 + 1 + 0=1
    \end{align*}
    Thus, the result.
\end{proof}
\begin{rem}
    Theorem \ref{proportion} points out that as the number of vertices grows, more often than not, graphs admit an eigenvector with a single 1 entry, a single -1 entry, and $n -2$ zero entries. In light of Theorem \ref{singleedge}, this means we are free to remove a substantial amount of edges without worrying about affecting the eigenvalue or corresponding eigenvector. What Theorem \ref{proportion} does not address is how many of these types of eigenvectors a given graph has. This may be an interesting place for exploration.
\end{rem} 
\section{Conclusion}
We have addressed some open problems regarding linear combinations of common graph-associated matrices. We also explore the eigenvectors of one of these matrices to open up discussion of the relationship between graph structure and eigenvectors of general graph-associated matrices.  Moreover, we classify the graphs that admit a specific type of eigenvector preserved under perturbation. Doing all of these using the general $B_{\alpha}$ matrix demonstrates the utility of the general theory.

\section*{Declaration of AI Usage}
No AI has been used in the completion of this work.
\newpage
\bibliographystyle{unsrt}  
\bibliography{references}  

\end{document}